\documentclass[draftcls,onecolumn]{IEEEtran}
\usepackage{cite}

\usepackage[cmex10]{amsmath}
\usepackage{array}
\usepackage{url}

\usepackage{amsfonts}
\usepackage{graphics}

 \newcommand{\Z}{\mathbb{Z}}
 
 \newcommand{\C}{{\cal C}}
  \newcommand{\D}{{\cal D}}
 \newcommand{\N}{{\mathbb{N}}}

\newtheorem{theorem}{Theorem}[section]

\newtheorem{corollary}[theorem]{Corollary}
\newtheorem{proposition}[theorem]{Proposition}
\newtheorem{lemma}[theorem]{Lemma}
\newtheorem{example}[theorem]{Example}

\begin{document}
%
\title{There is exactly one ${\mathbb{Z}}_2{\mathbb{Z}}_4$-cyclic $1$-perfect code}
%
%
%

\author{Joaquim Borges and Cristina Fern\'andez-C\'ordoba
\thanks{Manuscript received Month day, year; revised Month day, year.}
\thanks{J. Borges is with the Department of Information and Communications Engineering, Universitat Aut\`{o}noma de Barcelona, 08193-Bellaterra, Spain (e-mail: joaquim.borges@uab.cat)}
\thanks{C. Fern\'andez-C\'ordoba is with the Department of Information and Communications Engineering, Universitat Aut\`{o}noma de Barcelona, 08193-Bellaterra, Spain (e-mail: cristina.fernandez@uab.cat).}
\thanks{This work has been partially supported by the Spanish MICINN grant TIN2013-40524-P and by the Catalan AGAUR grant 2014SGR-691.}
}

\markboth{Journal of \LaTeX\ Class Files,~Vol.~XX, No.~Y, Month~year}%
{Shell \MakeLowercase{\textit{et al.}}: Bare Demo of IEEEtran.cls for Journals}

\maketitle

\begin{abstract}
 Let $\C$ be a ${\mathbb{Z}}_2{\mathbb{Z}}_4$-additive code of length $n > 3$. We prove that if the binary Gray image of $\C$, $C=\Phi(\C)$, is a 1-perfect nonlinear code, then $\C$ cannot be a ${\mathbb{Z}}_2{\mathbb{Z}}_4$-cyclic code except for one case of length $n=15$. Moreover, we give a parity check matrix for this cyclic code. Adding an even parity check coordinate to a ${\mathbb{Z}}_2{\mathbb{Z}}_4$-additive 1-perfect code gives an extended 1-perfect code. We also prove that any such code cannot be ${\mathbb{Z}}_2{\mathbb{Z}}_4$-cyclic.
\end{abstract}

\begin{IEEEkeywords}
Perfect codes, ${\mathbb{Z}}_2{\mathbb{Z}}_4$-additive cyclic codes, simplex codes.
\end{IEEEkeywords}

\IEEEpeerreviewmaketitle

\section{Introduction}


A $\Z_2\Z_4$-linear code $C$ is the binary Gray image of a $\Z_2\Z_4$-additive code $\C\subseteq\Z_2^\alpha\times\Z_4^\beta$, and if $\beta=0$, then $C$ is a binary linear code. If $\alpha=0$, then $C$ is called $\Z_4$-linear. In 1997, a first family of $\Z_2\Z_4$-linear 1-perfect codes was presented in \cite{PuRi} in the more general context of translation-invariant propelinear codes. Lately, in 1999, all $\Z_2\Z_4$-linear 1-perfect codes were fully classified in \cite{Add}. Specifically, for every appropriate values of $\alpha$ and $\beta$, there exists exactly one $\Z_2\Z_4$-linear 1-perfect code $C$. Note that when $\beta=0$, then $C$ is a Hamming code. In subsequent papers (\cite{RankKernel} and \cite{Kro}), $\Z_2\Z_4$-linear extended 1-perfect codes were also classified. But it was not until 2010, when an exhaustive description of general $\Z_2\Z_4$-linear codes appeared \cite{AddDual}. More recently, in 2014, $\Z_2\Z_4$-cyclic codes have been defined in \cite{Abu}, and also studied in \cite{Z2Z4Dual}.

After all these papers, a natural question is to ask for the existence or nonexistence of $\Z_2\Z_4$-cyclic 1-perfect codes, of course, excluding the linear (Hamming) case when $\beta=0$. In this paper, we show that such codes do not exist with only one exception. This unique $\Z_2\Z_4$-cyclic 1-perfect code has binary length 15, with $\alpha=3$ and $\beta=6$. We also give a parity check matrix for such code. If we add an even parity check coordinate to a $\Z_2\Z_4$-linear 1-perfect code, then we obtain a $\Z_2\Z_4$-linear extended 1-perfect code. We show that none of these codes can be $\Z_2\Z_4$-cyclic.

The paper is organized as follows. In the next section, we give basic definitions and properties. Moreover, we give the type of all $\Z_2\Z_4$-linear 1-perfect codes, computing some parameters that were not specified in \cite{Add}. In Section \ref{principal}, we give the main results of this paper. First, we prove that in a $\Z_2\Z_4$-cyclic 1-perfect code, $\beta$ must be a multiple of $\alpha$. This, immediately excludes a lot of cases. For the remaining ones, using a key property of simplex codes, we prove that $\alpha$ cannot be greater than 3. Therefore, finally, we have only one possible case when $\alpha=3$ and $\beta=6$. In Example \ref{exemple}, we give a parity check matrix for this code in a cyclic form. In Section \ref{estesos}, we prove that a $\Z_2\Z_4$-linear extended 1-perfect code, with $\alpha > 0$, cannot be $\Z_2\Z_4$-cyclic.

\section{Preliminaries}

Denote by ${\mathbb{Z}}_2$ and ${\mathbb{Z}}_4$ the rings of integers modulo 2 and modulo 4,
respectively. A binary code of length $n$ is any non-empty subset $C$ of ${\mathbb{Z}}_2^n$. If that
subset is a vector space then we say that it is a linear code.  Any
non-empty subset ${\cal C}$
of ${\mathbb{Z}}_4^n$ is a quaternary code of length $n$, and an additive subgroup of ${\mathbb{Z}}_4^n$ is
called a quaternary linear code. The elements of a code are usually called codewords.

Given two binary vectors $u,v\in\Z_2^n$, the (Hamming) distance between $x$ and $y$, denoted $d(u,v)$, is the number of coordinates in which they differ. The (Hamming) weight of any vector $z\in\Z_2^n$, $w(z)$, is the number of nonzero coordinates of $z$. The Lee weights of $0, 1, 2, 3\in\Z_4$ are $0, 1, 2, 1$ respectively, and the Lee weight of $a\in\Z_4^m$, $w_L(a)$, is the rational sum of the Lee weights of its components. If $a,b\in\Z_4^m$, then the Lee distance between $a$ and $b$ is $d_L(a,b)=w_L(a-b)$. For a vector ${\bf  u} \in {\mathbb{Z}}_2^\alpha\times{\mathbb{Z}}_4^\beta$ we write
${\bf  u}=(u\mid  u')$ where $u\in\Z_2^\alpha$ and $u'\in\Z_4^\beta$. The weight of ${\bf  u}$ is $w({\bf  u})=w(u)+w_L(u')$. If ${\bf  u}, {\bf v}\in {\mathbb{Z}}_2^\alpha\times{\mathbb{Z}}_4^\beta$, the distance between ${\bf u}=(u\mid  u')$ and ${\bf v}=(v\mid  v')$ is defined as $d({\bf u}, {\bf v})=d(u,v)+d_L(u',v')$. The classical Gray map $\phi:\;\Z_4\;\longrightarrow\;\Z_2^{2}$ is defined by
$$
\phi(0)=(0,0),\;\;\phi(1)=(0,1),\;\;\phi(2)=(1,1),\;\;\phi(3)=(1,0).
$$
If $a=(a_1,\ldots,a_m)\in\Z_4^m$, then the Gray map of $a$ is the coordinatewise extended map $\phi(a)=(\phi(a_1),\ldots,\phi(a_m))$. We naturally extend the Gray map for vectors ${\bf u}=(u\mid  u')\in {\mathbb{Z}}_2^\alpha\times{\mathbb{Z}}_4^\beta$ so that $\Phi({\bf u}) = (u\mid \phi(u'))$. Clearly, the Gray map transforms Lee distances and weights to Hamming distances and weights. Hence, if ${\bf u}, {\bf v}\in {\mathbb{Z}}_2^\alpha\times{\mathbb{Z}}_4^\beta$, we have that $d({\bf u}, {\bf v})=d(\Phi({\bf u}), \Phi({\bf v}))$.

A binary code $C$ of length $n$ is called 1-perfect if any vector not in $C$ is at distance one from exactly one codeword in $C$. Such codes have minimum distance 3 between any pair of codewords, and the cardinality is $|C|=2^n/(n+1)$. It is well known that $n=2^t-1$, for some $t\geq 2$ and hence $|C|=2^{2^t-t-1}$. For any $t$, there is exactly one linear 1-perfect code, up to coordinate permutation, which is called the Hamming code.
An extended 1-perfect code $C'$ is obtained by adding an even parity check coordinate to a 1-perfect code $C$. In this case, $C'$ has minimum distance 4, length $n+1=2^t$, and size $|C'|=2^{2^t-t-1}$.

The dual of a binary Hamming code is a constant weight code called {\em simplex}. The dual of an extended Hamming code is a linear {\em Hadamard} code. In this paper, we make use of two important properties \cite{MacW,Just}:
\begin{itemize}
\item[(a)] A binary Hamming code is cyclic, that is, its coordinates can be arranged such that the cyclic shift of any codeword is again a codeword. Therefore, simplex codes are also cyclic.
\item[(b)] An extended Hamming code of length greater than 4 is not cyclic. Hence, a linear Hadamard code of length greater than 4 is not cyclic.
\end{itemize}

A $\mathbb{Z}_2\mathbb{Z}_4$-additive code $\C$ is an additive subgroup of $\mathbb{Z}_2^\alpha\times\mathbb{Z}_4^\beta$. Such codes are extensively studied in \cite{AddDual}.
Since ${\cal C}$ is a subgroup of $\mathbb{Z}_2^\alpha\times\mathbb{Z}_4^\beta$, it is also
isomorphic to a group $\mathbb{Z}_2^\gamma\times\mathbb{Z}_4^\delta$.
Therefore, ${\cal C}$ is of type $2^\gamma 4^\delta$ as a group, it has $|{\cal
C}| = 2^{\gamma +2\delta}$ codewords, and the number of codewords of order less than two in
${\cal C}$ is $2^{\gamma +\delta}$.

Let $X$ (respectively $Y$) be the set of $\mathbb{Z}_2$ (respectively
$\mathbb{Z}_4$) coordinate positions, so $|X| =\alpha$ and $|Y| = \beta$. Unless
otherwise stated, the set $X$ corresponds to the first $\alpha$ coordinates and
$Y$ corresponds to the last $\beta$ coordinates. Call ${\cal C}_X$
(respectively ${\cal C}_Y )$ the punctured code of ${\cal C}$ by deleting the
coordinates outside $X$ (respectively $Y$), and removing repeated codewords, if necessary. Let ${\cal C}_b$ be the subcode of
${\cal C}$ which contains all order two codewords and the zero codeword. Let $\kappa$ be the
dimension of $({\cal C}_b)_X$, which is a binary linear code.

According to \cite{AddDual}, and considering all these parameters, we say that ${\cal C}$ is a $\Z_2\Z_4$-additive code of type
$(\alpha, \beta; \gamma , \delta; \kappa)$. The binary Gray image of $\C$ is $C=\Phi(\C)=\{\Phi({\bf x})\mid {\bf x}\in\C\}$. In this case, $C$ is called a $\Z_2\Z_4$-linear code of type $(\alpha, \beta; \gamma , \delta; \kappa)$ and its length is $n=\alpha+2\beta$.

The standard inner product in $\mathbb{Z}_2^\alpha\times\mathbb{Z}_4^\beta$, defined in \cite{AddDual}, can be written as
$$\textbf{u}\cdot \textbf{v} = 2\left(\sum_{i=1}^{\alpha}u_iv_i\right)+\sum_{j=1}^{\beta}u'_jv'_j\in \mathbb{Z}_4,$$
where the computations are made taking the zeros and ones in the $\alpha$ binary coordinates as quaternary zeros and ones, respectively. The dual code of ${\cal C}$, is defined in the standard way by
$${\cal C}^\perp=\{\textbf{v} \in \mathbb{Z}_2^\alpha\times\mathbb{Z}_4^\beta \mid \textbf{u}\cdot\textbf{v}=0, \mbox{ for all }\textbf{u}\in{\cal C}\}.$$

The types of dual codes are related in \cite{AddDual}.

\begin{proposition}[\cite{AddDual}]\label{ParamsDual}
If $\C$ is a $\Z_2\Z_4$-additive code of type $(\alpha,\beta;\gamma,\delta;\kappa)$, then its dual code $\C^\perp$ is of type
$$
(\alpha,\beta;\alpha+\gamma-2\kappa,\beta-\gamma-\delta+\kappa; \alpha-\kappa).
$$
\end{proposition}

Let $C$ be a $\Z_2\Z_4$-linear 1-perfect code. Then, the corresponding $\Z_2\Z_4$-additive code $\Phi^{-1}(C)$ is also called 1-perfect code. Such codes are completely characterized.

\begin{proposition}[\cite{Add}]\label{Caracteritzacio}
$\mbox{ }$

\begin{itemize}
\item[(i)] Let $n=2^t-1$, where $t\geq 4$. Then, for every $r$ such that $2\leq r\leq t\leq 2r$, there is exactly one $\Z_2\Z_4$-linear 1-perfect code of length $n$, up to coordinate permutation, with parameters $\alpha=2^r-1$ and $\beta= 2^{t-1}-2^{r-1}$.
\item[(ii)] There are no other $\Z_2\Z_4$-linear 1-perfect codes.
\end{itemize}
\end{proposition}

\bigskip

Here, we strength a little this result by computing the type of these codes. Since $r$ and $t$ completely determine a $\Z_2\Z_4$-linear 1-perfect code, we denote such code by $C_{r,t}$. The corresponding $\Z_2\Z_4$-additive code is $\C_{r,t}=\Phi^{-1}(C_{r,t})$.

\begin{proposition}\label{parametres}
Let $\C_{r,t}$ be of type $(\alpha,\beta;\gamma,\delta;\kappa)$ and let $(\C_{r,t})^\perp$ be the dual code of type $(\bar{\alpha},\bar{\beta};\bar{\gamma},\bar{\delta};\bar{\kappa})$. Then,
\begin{itemize}
\item[(i)] The parameters of $\C_{r,t}$ are:
\begin{eqnarray*}
\alpha &=& 2^r-1; \;\;\; \beta=2^{t-1}-2^{r-1}; \\
\gamma &=& 2^r-1-2r+t; \\
\delta &=& 2^{t-1}-2^{r-1}+r-t;\\
\kappa &=& \gamma.
\end{eqnarray*}
\item[(ii)] The parameters of $(\C_{r,t})^\perp$ are:
\begin{eqnarray*}
\bar{\alpha} &=& \alpha; \;\;\; \bar{\beta}=\beta; \\
\bar{\gamma} &=& 2r-t; \;\;\; \bar{\delta}=t-r; \\
\bar{\kappa} &=& \bar{\gamma}.
\end{eqnarray*}
\end{itemize}
\end{proposition}

\begin{IEEEproof}
The parameters $\alpha$, $\beta$, $\bar{\alpha}$ and $\bar{\beta}$ follow directly from Proposition \ref{Caracteritzacio}.

On the one hand, the binary linear code $C_0=\{(x\mid 0,\ldots,0)\in \C_{r,t}\}_X$ is clearly 1-perfect, i.e. a Hamming code. Hence, $C_0$ has dimension $2^r-r-1$. This means that the zero codeword in $(\C_{r,t})_Y$ (and any other one) is repeated $2^{2^r-r-1}$ times in $\C_{r,t}$. On the other hand, consider a vector of the form
$$
{\bf u}=(u\mid u')=(0,\ldots,0\mid 0,\ldots 0,2,0,\ldots,0)\in\Z_2^\alpha\times\Z_4^\beta,
 $$
 where $\alpha=2^r-1$ and $\beta=2^{t-1}-2^{r-1}$. Since the minimum distance in $\C_{r,t}$ is 3, the minimum weight is also 3 because $\C_{r,t}$ is distance invariant \cite{PuRi}. Hence ${\bf u}$ must be at distance one from a weight 3 codeword ${\bf x}=(x\mid x')$, where $w(x)=1$ and $x'=u'$. Indeed, if $w(x)=0$ and $w(x')=3$, then $2{\bf x}$ would have weight 2. Therefore, $(\C_{r,t})_Y$ has $2^\beta$ distinct codewords of order two (including here the zero codeword). We conclude that $\C_{r,t}$ has $2^\beta\cdot 2^{2^r-r-1}$ order two codewords (again, including the zero codeword). Thus, the dimension of $(\C_{r,t})_b$ is
\begin{equation}\label{eq1}
\gamma +\delta=\beta + 2^r-r-1=2^{t-1} + 2^{r-1} -r -1.
\end{equation}

The size of $\C_{r,t}$ is $2^{2^t-t-1}$. Therefore,
\begin{equation}\label{eq2}
\gamma+2\delta=2^t-t-1.
\end{equation}

Combining Equations \ref{eq1} and \ref{eq2}, we obtain the values of $\gamma$ and $\delta$.

As can be seen in \cite{Add}, the quotient group $\Z_2^\alpha\times\Z_4^\beta / \C_{r,t}$ is isomorphic to $\Z_2^{2r-t}\times \Z_4^{t-r}$. In other words, $\C_{r,t}^\perp$ has parameters $\bar{\gamma}=2r-t$ and $\bar{\delta}=t-r$. Now, the values of $\kappa$ and $\bar{\kappa}$ are easily obtained by applying Proposition \ref{ParamsDual}.
\end{IEEEproof}

\bigskip

Let $v=(v_1,\ldots,v_m)$ be an element in $\Z_2^m$ or $\Z_4^m$. We denote by $\sigma(v)$ the right cyclic shift of $v$, i.e. $\sigma(v)=(v_m,v_1,\ldots,v_{m-1})$.
We recursively define $\sigma^j(v)=\sigma\left(\sigma^{j-1}(v)\right)$, for $j=2,3,\ldots$ For vectors ${\bf  u} = ( u \mid  u') \in {\mathbb{Z}}_2^\alpha \times {\mathbb{Z}}_4^\beta$ we extend the definition of $\sigma$ as the double right cyclic shift of ${\bf  u}$, that is, $\sigma({\bf u})=\left(\sigma(u)\mid \sigma(u')\right)$.

A ${\mathbb{Z}_2 {\mathbb{Z}_4}}$-additive code ${\cal C}\subseteq\mathbb{Z}_2^\alpha\times\mathbb{Z}_4^\beta$ is a $\Z_2\Z_4$-cyclic code if for each codeword $\textbf{x}\in {\cal C}$, we have that $\sigma({\bf x})\in\C$. Such codes were first defined in \cite{Abu} and also studied in \cite{Z2Z4Dual}. As can be seen in \cite{Abu}, the dual of a $\Z_2\Z_4$-cyclic code is also $\Z_2\Z_4$-cyclic.

\section{There is no nontrivial $\Z_2\Z_4$-cyclic perfect codes with one exception}\label{principal}

We say that a code is nontrivial if it has more than two codewords and its minimum distance is $d>1$. Apart from 1-perfect codes, there is only another nontrivial binary perfect code. It is the linear binary Golay code of length 23. But this code has not any $\Z_2\Z_4$-linear structure apart from the binary linear one \cite{Josep}. Therefore, any binary nonlinear and nontrivial $\Z_2\Z_4$-linear perfect code is a 1-perfect code.

In this section, we prove that for any $\Z_2\Z_4$-linear 1-perfect code, which is not a Hamming code, its corresponding $\Z_2\Z_4$-additive code cannot be $\Z_2\Z_4$-cyclic with exactly one exception.

\begin{proposition}\label{multiple}
If $\C_{r,t}$ is a $\Z_2\Z_4$-cyclic 1-perfect code, then $t=r$ or $t=2r$.
\end{proposition}

\begin{IEEEproof}
By the argument in the proof of Proposition \ref{parametres}, we may assume that $\C_{r,t}$ contains a codeword of the form ${\bf x}=(x\mid 2,0,\ldots,0)$ with $w(x)=1$. Now, consider the codeword ${\bf z}=\sigma^\beta({\bf x})$. If ${\bf z}\neq {\bf x}$ then ${\bf z} + {\bf x}$ would have weight 2. Consequently, ${\bf z}$ must be equal to ${\bf x}$ implying that $\beta$ is a multiple of $\alpha$, that is, $2^{t-1}-2^{r-1}$ is a multiple of $2^r-1$. Thus,
$$
\frac{2^{r-1}(2^{t-r}-1)}{2^r-1}\in\N \;\Longrightarrow\; \frac{(2^{t-r}-1)}{2^r-1}\in\N.
$$
Therefore $r$ divides $t-r$ implying that $r$ divides $t$. Since $r\leq t\leq 2r$, the only possibilities are $t=r$ or $t=2r$.
\end{IEEEproof}

\bigskip

If $t=r$, then $C_{r,t}=\Phi(\C_{r,t})$ is linear, i.e. a Hamming code. In effect, it is well known that its coordinates can be arranged such that it is a binary cyclic code. We are interested in those codes whose binary Gray image is not linear, that is, when $t=2r$. For this case, $t=2r$, we have that $\C_{r,2r}$ is of type
$$
(2^r-1,2^{r-1}(2^r-1);2^r-1,2^{r-1}(2^r-1)-r;2^r-1),
$$
and applying Proposition \ref{parametres} we obtain that its dual code $\C^\perp_{r,2r}$ is of type
$$
(2^r-1,2^{r-1}(2^r-1);0,r;0).
$$

\begin{example}\label{exemple}
For $r=2$ we have that the type of $\C_{2,4}$ is $(3,6;3,4;3)$. By Proposition \ref{parametres}, its dual code $\C_{2,4}^\perp$ is of type $(3,6;0,2;0)$. Consider the matrix
$$
H=\left(
    \begin{array}{ccc|cccccc}
      1 & 1 & 0 & 1 & 1 & 2 & 3 & 1 & 0 \\
      0 & 1 & 1 & 0 & 1 & 1 & 2 & 3 & 1 \\
    \end{array}
  \right).
$$

The matrix $H$ generates a code of type $(3,6;0,2;0)$. Any column is not a multiple of another one. Hence the code $\C^*$ with parity check matrix $H$ has minimum distance at least 3, type $(3,6;3,4;3)$ and size $2^{11}$. Therefore,
$\C^*$ is the $\Z_2\Z_4$-additive 1-perfect code $\C_{2,4}$ and $H$ generates $\C_{2,4}^\perp$. Note that the second row of $H$ is the shift of the first one. Also, the first row minus the second one gives the shift of the second row. Since the shift of any row of $H$ is a codeword, we have that the shift of any codeword is again a codeword. Consequently, $\C_{2,4}^\perp$ is a $\Z_2\Z_4$-cyclic code and so is $\C_{2,4}$.
\end{example}

\bigskip

From now on, we denote by $D^{(r)}$ the code $\C^\perp_{r,2r}$ of binary length $n=\alpha+2\beta=2^{2r}-1$. Hence, $D^{(r)}_b$ is the set of codewords of order 2 and the zero codeword. Recall that the dual of a binary Hamming code is called simplex. Of course, the coordinates of a simplex code can be arranged such that the code is cyclic. We denote by $S_r$ a cyclic simplex code of length $2^r-1$.

\begin{lemma}\label{PesConstant}
The code $D^{(r)}$ is a constant weight code, where all nonzero codewords have weight $2^{2r-1}$.
\end{lemma}

\begin{IEEEproof}
The weight distributions of dual codes are related by the MacWilliams identity \cite{PuRi,DelLev}, as well as for binary linear codes. It is well known that any 1-perfect code has the same weight distribution as the Hamming code of the same length. Therefore, $D^{(r)}$ must have the same weight distribution as the simplex code of length $n=2^{2r}-1$. Hence, the weight of any nonzero codeword is $(n+1)/2=2^{2r-1}$.
\end{IEEEproof}

\bigskip

\begin{proposition}\label{zeros}
If $D^{(r)}$ is $\Z_2\Z_4$-cyclic, then $(D^{(r)})_X = S_r$. Moreover, a codeword ${\bf z}\in D^{(r)}$ has the zero vector in the $\Z_2$ part, ${\bf z}=(0,\ldots,0\mid z_1',\ldots,z'_\beta)$, if and only if ${\bf z}\in D^{(r)}_b$.
\end{proposition}

\begin{IEEEproof}
A generator matrix for $D^{(r)}$ would have the form
$$
G=\left(
    \begin{array}{c|c}
      G_1 & G_2  \\
    \end{array}
  \right),
$$
where $G_1$ is a $r \times 2^r-1$ generator matrix for $(D^{(r)})_X$. Since the minimum weight of $\C_{r,2r}$ is 3, $G_1$ has neither repeated columns, nor the zero column. Therefore $G_1$ has as columns all the nonzero binary vectors of length $r$ and $(D^{(r)})_X=S_r$.
The size of $D^{(r)}$ is $|D^{(r)}|=2^{2r}$ and
the number of codewords of order less than or equal to 2 is $|D^{(r)}_b|=2^r$. Hence, $D^{(r)}$ can be viewed as a set of $2^r$ cosets of $D^{(r)}_b$. We conclude that each codeword in $(D^{(r)})_X$ appears $2^r$ times in $D^{(r)}$. So, the zero codeword in $(D^{(r)})_X$ appears in $D^{(r)}$ exactly in the codewords of $D^{(r)}_b$.
\end{IEEEproof}

\bigskip

%
%

\begin{proposition}\label{SonSimplexs}
Suppose that $D^{(r)}$ is $\Z_2\Z_4$-cyclic. If we change the coordinates `2' by `1' in $(D^{(r)}_b)_Y$ we obtain $2^{r-1}$ copies of $S_r$.
\end{proposition}

\begin{IEEEproof}
Clearly, when we change the twos by ones in $(D^{(r)}_b)_Y$, we obtain a binary linear cyclic code $D$ with constant weight and dimension $r$. By \cite{Boni}, $D$ must be a simplex code or a replication of a simplex code. Since the dimension is $r$, we conclude that $D$ is a replication of a simplex code of length $2^r-1$. Moreover, since $(D^{(r)}_b)_Y$ is cyclic, $D$ is a replication of $S_r$.
\end{IEEEproof}

\bigskip

Therefore, if $D^{(r)}$ is $\Z_2\Z_4$-cyclic, any order 4 codeword is of the form:
$$
{\bf z}=(x_1,\ldots,x_\alpha\mid y^{(1)},\ldots,y^{(2^{r-1})}),
$$
where $y^{(i)}=(y^{(i)}_1,\ldots,y^{(i)}_\alpha)$, for all $i=1,\ldots,2^{r-1}$. The set of coordinate positions of $y^{(i)}$ will be called the $i$th {\em block}. Taking into account that $2{\bf z}\in D^{(r)}_b$ and by Proposition \ref{SonSimplexs}, we see that ${\bf z}$
has $2^{r-1}$ odd coordinates (i.e. coordinates from $\{1,3\}$) in any block at the same positions. In other words, $y^{(i)}\equiv y^{(j)} \pmod{2}$, for all $i,j=1,\ldots,2^{r-1}$.

\begin{corollary}\label{PesosFiles}
Let ${\bf z}=(x_1,\ldots,x_\alpha\mid y^{(1)},\ldots,y^{(2^{r-1})})\in D^{(r)}$ be an order 4 codeword and assume that $D^{(r)}$ is $\Z_2\Z_4$-cyclic. Then, $(y^{(1)},\ldots,y^{(2^{r-1})})$ has:
\begin{eqnarray*}
2^{2r-2} \;\;\;\; & \mbox{ odd coordinates} \\
2^{r-2}(2^{r-1}-1) \;\;\;\; & \mbox{ twos, and} \\
2^{r-2}(2^{r-1}-1) \;\;\;\; & \mbox{ zeroes.}
\end{eqnarray*}
\end{corollary}

\begin{IEEEproof}
The result follows from Lemma \ref{PesConstant}, Proposition \ref{zeros} and Proposition \ref{SonSimplexs}.
\end{IEEEproof}

\bigskip

For any binary vector $x=(x_1,\ldots,x_m)$, the support of $x$ is the set of nonzero positions, $supp(x)=\{i\mid x_i\neq 0\}$. Note that $w(x)=|supp(x)|$. We define $\overline{supp}(x)=\{1,\ldots,m\}\setminus supp(x)$ as the complementary support of $x$.

\begin{lemma}\label{simplex}
Let $S_r$ be a cyclic simplex code of length $2^r-1$, with $r>2$. For any pair of codewords $x,y\in S_r$ we have that $|supp(x) \cap \overline{supp}(y)|$ is even. In other words, $x$ cannot have an odd number of nonzero positions in $\overline{supp}(y)$.
\end{lemma}

\begin{IEEEproof}
The distance between $x$ and $y$ must be $2^{r-1}$. Therefore,
$$
d(x,y)=|supp(x)|+|supp(y)|-2|supp(x)\cap supp(y)|=2^{r-1}.
$$
But the weight of any codeword is $2^{r-1}$. Thus,
$$
2^{r-1} + 2^{r-1} - 2|supp(x)\cap supp(y)| = 2^{r-1},
$$
implying that $|supp(x)\cap supp(y)|=2^{r-2}$, which is even for $r>2$. Hence, $|supp(x)\cap \overline{supp}(y)|$ is also even for $r>2$.
\end{IEEEproof}

\begin{proposition}\label{NoCoincidencies}
Suppose that $D^{(r)}$ is $\Z_2\Z_4$-cyclic and $r>2$. Let ${\bf z}=(x_1,\ldots,x_\alpha\mid y^{(1)},\ldots,y^{(2^{r-1})})\in D^{(r)}$ be an order 4 codeword. For any distinct $i,j$, define
$$
N_{i,j}=\{\ell\mid 1\leq \ell\leq\alpha,\;y^{(i)}_\ell,y^{(j)}_\ell\in\{0,2\},\;y^{(i)}_\ell\neq y^{(j)}_\ell\},
$$
i.e. $N_{i,j}$ is the set of coordinate positions where $y^{(i)}$ has a `2' and $y^{(j)}$ has `0' or vice versa. Then, $|N_{i,j}|$ is even.
\end{proposition}

\begin{IEEEproof}
Suppose to the contrary that $|N_{i,j}|$ is odd. Assume that $i<j$ and consider the codeword ${\bf v}=\sigma^{\alpha(j-i)}({\bf z})$. Clearly, ${\bf u}={\bf v} + {\bf z}$ has the zero vector in the $\Z_2$ part. Thus, by Proposition \ref{zeros}, ${\bf u}$ is an order two codeword. Now, comparing with the codeword $2{\bf v}$ (or $2{\bf z}$), we can see that ${\bf u}$ has an odd number of twos in $\overline{supp}(2{\bf v})$ in the $j$th block, contradicting Lemma \ref{simplex}.
\end{IEEEproof}

\bigskip

As a consequence, we obtain that in any order 4 codeword, the number of twos in any block has the same parity.

\begin{corollary}\label{paritats}
Suppose that $D^{(r)}$ is $\Z_2\Z_4$-cyclic and $r>2$. Let $(x_1,\ldots,x_\alpha\mid y^{(1)},\ldots,y^{(2^{r-1})})\in D^{(r)}$ be an order 4 codeword. Put $\eta_k(y)=|\{\ell\mid 1\leq\ell\leq\alpha,\; y^{(k)}_\ell =2\}|$. Then, $\eta_1(y),\ldots,\eta_{2^{r-1}}(y)$ have all the same parity.
\end{corollary}

\begin{IEEEproof}
Straightforward from Proposition \ref{NoCoincidencies}.
\end{IEEEproof}

\bigskip

\begin{lemma}\label{dosos}
Suppose that $D^{(r)}$ is $\Z_2\Z_4$-cyclic and $r>2$. As before, let ${\bf z}=(x_1,\ldots,x_\alpha\mid y^{(1)},\ldots,y^{(2^{r-1})})\in D^{(r)}$ be an order 4 codeword. Then, there exist different $k,k'\in\{1,\ldots,2^{r-1}\}$ such that $\eta_k(y) \neq \eta_{k'}(y)$. Moreover, if for some $\ell\in\{1,\ldots,\alpha\}$ we have $y^{(k)}_\ell=0$ and $y^{(k')}_\ell=2$, then
\begin{eqnarray*}
|\{i\mid 1\leq i \leq 2^{r-1},\;y^{(i)}_\ell=0\}| &=& \\
|\{j\mid 1\leq j \leq 2^{r-1},\;y^{(j)}_\ell=2\}| &=& 2^{r-2}.
\end{eqnarray*}
\end{lemma}

\begin{IEEEproof}
The total number of twos in ${\bf z}$ is $2^{r-2}(2^{r-1}-1)$ (see Corollary \ref{PesosFiles}). But this number is not divisible by $2^{r-1}$ and hence not all the blocks have the same number of twos. This proves that $\eta_k(y) \neq \eta_{k'}(y)$ for some $k,k'\in\{1,\ldots,2^{r-1}\}$.

Let $k$ and $k'=k+1$ be such that $\eta_k(y) \neq \eta_{k'}(y)$. Without loss of generality, we assume that $k'=2^{r-1}$ and $k=2^{r-1}-1$. After some shifts of ${\bf z}$, we can get the situation that $y^{(k)}_\alpha\neq y^{(k')}_\alpha$, where $y^{(k)}_\alpha,y^{(k')}_\alpha\in\{0,2\}$. That is, the last coordinates of the last two blocks are in $\{0,2\}$ and different from each other. Now, if we shift the codeword, $\eta_{2^{r-1}}(y)$ changes its parity. Hence, by Corollary \ref{paritats}, $\eta_{2^{r-1}-1}(y)$ must change its parity as well, implying that $y^{(2^{r-1}-2)}_\alpha \neq y^{(2^{r-1}-1)}_\alpha$ and $y^{(2^{r-1}-2)}_\alpha, y^{(2^{r-1}-1)}_\alpha \in\{0,2\}$. With the same argument, $y^{(2^{r-1}-3)}_\alpha \neq y^{(2^{r-1}-2)}_\alpha$, $y^{(2^{r-1}-3)}_\alpha, y^{(2^{r-1}-2)}_\alpha\in\{0,2\}$, and so on. Therefore, in this last coordinate, half of the blocks have a `0' and half of the blocks have a `2'.
\end{IEEEproof}

\bigskip

Now, we are ready to prove the nonexistence of a $\Z_2\Z_4$-cyclic code $D^{(r)}$ for $r>2$.

\begin{theorem}
There is no $\Z_2\Z_4$-cyclic 1-perfect code $\C$ such that $C=\Phi(\C)$ is nonlinear except for the case when $\C=\C^*$ is the code of Example \ref{exemple} of type $(3,6;3,4;3)$, which is a $\Z_2\Z_4$-cyclic code.
\end{theorem}

\begin{IEEEproof}
Assume that $\C$ is a $\Z_2\Z_4$-cyclic 1-perfect code such that $C=\Phi(\C)$ is nonlinear. By Proposition \ref{multiple}, $\C$ must be a code $\C_{r,2r}$. If $r=2$, then we have seen the $\Z_2\Z_4$-cyclic code $\C^*=\C_{2,4}$ in Example \ref{exemple}. Suppose now that $r>2$.

Let ${\bf z}=(x_1,\ldots,x_\alpha\mid y^{(1)},\ldots,y^{(2^{r-1})})\in \C^\perp$ be an order 4 codeword.
Define
\begin{eqnarray*}
\lambda &=& \left|\left\{\ell\mid 1\leq\ell\leq\alpha,\; y^{(i)}_\ell=2, \;\;\forall i=1,\ldots,2^{r-1}\right\}\right|, \mbox{ and }\\
\mu     &=& \left|\left\{\ell\mid 1\leq\ell\leq\alpha,\mbox{ such that } \exists\; k,k' \mbox{ with } y^{(k)}_\ell\neq y^{(k')}_\ell;\right.\right. \\
 &&\left.\left. \;\;\;\;y^{(k)}_\ell,y^{(k')}_\ell\in\{0,2\}\right\}\right|.
\end{eqnarray*}

Then, by Lemma \ref{dosos}, the number of twos in ${\bf z}$ is
$2^{r-1}\lambda + 2^{r-2}\mu$. We have seen in Corollary \ref{PesosFiles} that this must equal $2^{r-2}(2^{r-1}-1)$. Thus, we obtain
$$
2\lambda + \mu = 2^{r-1}-1,
$$
implying that $\mu$ is an odd number. But this is a contradiction with Proposition \ref{NoCoincidencies}.
\end{IEEEproof}

\section{The nonexistence of nontrivial $\Z_2\Z_4$-cyclic extended perfect codes}\label{estesos}

Given a $\Z_2\Z_4$-additive 1-perfect code $\C_{r,t}$ ($2\leq r\leq t\leq 2r$), we denote by $\C_{r,t}'$ the extended code obtained by adding an even parity check coordinate (of course, at the $\Z_2$ part). Then, $\C_{r,t}'$ is a $\Z_2\Z_4$-additive extended 1-perfect code. Recall that $\C_{r,t}$ is of type
\small
$$
(2^r-1,2^{t-1}-2^{r-1}; 2^r-1-2r+t, 2^{t-1}-2^{r-1}+r-t; 2^r-1-2r+t).
$$
\normalsize
Since $|\C_{r,t}'|=|\C_{r,t}|$, $|(\C_{r,t}')_b|=|(\C_{r,t})_b|$, and $|((\C_{r,t}')_b)_X|=|((\C_{r,t}')_b)_X|$, we have that $\C_{r,t}'$ is of type
\small
$$
(2^r,2^{t-1}-2^{r-1}; 2^r-1-2r+t, 2^{t-1}-2^{r-1}+r-t; 2^r-1-2r+t).
$$
\normalsize

In this section, we prove that $\C_{r,t}'$ is not $\Z_2\Z_4$-cyclic for $t>2$. For this, we begin examining the case $r=2$. In such case, we have $t\in\{2,3,4\}$. The case $t=r=2$ corresponds to a binary linear cyclic code of length 4 and two codewords. Such code is the trivial repetition code of length 4. Hence, we consider the cases $t=3$ and $t=4$.

\begin{lemma}\label{DosConcrets}
The codes $\C_{2,3}'$ and $\C_{2,4}'$ are not $\Z_2\Z_4$-cyclic.
\end{lemma}

\begin{proof}
First, we consider the code $\C_{2,3}'$. The type of $\C_{2,3}'$ is $(4,2;2,1;2)$. Hence, $\C_{2,3}'$ contains 8 codewords of order 4. Let ${\bf x}=(x\mid x_1',x_2')$ be one such codeword. Since any codeword in $\C_{2,3}'$ has weight 4 or 8, it follows that $x_1'$ and $x_2'$ must be both odd coordinates (otherwise $2{\bf x}$ would have weight 2). Also, we have that $w(x)=2$. If we consider the codeword ${\bf x} + \sigma({\bf x})$, we can see that $x+\sigma(x)$ must have weight 4, implying that $x=(1,0,1,0)$ (or $x=(0,1,0,1)$). Now, take a codeword ${\bf y}=(y\mid y_1',y_2')$ such that $y_1'=x_1'$ and $y_2'\neq x_2'$ (a simple counting argument shows that exactly half of the codewords have equal the last two coordinates). We have that $d(x,y)\in\{0,4\}$ and hence $d({\bf x},{\bf y})\in\{2,6\}$, a contradiction.

The code $\C_{2,4}'$ is an extension of the code $\C^*$ in Example \ref{exemple}. Consider the dual code $\D=(\C_{2,4}')^\perp$. If $H$ is a generator matrix for $\C_{2,4}^\perp$, then a generator matrix for $\D$ can be obtained adding, first, a zero column to $H$ and, second, the row ${\bf f}=(1,\ldots,1\mid 2,\ldots,2)$. Hence, $\D$ is of type $(4,6;1,2;1)$ and any nonzero codeword ${\bf z}\neq {\bf f}$ has weight 8.  Let ${\bf x}$ be an order 4 codeword. Clearly, ${\bf x}$ must have 4 odd coordinates in the quaternary part (otherwise, $2{\bf x}$ would not have weight 8). This implies that ${\bf z}={\bf x}+\sigma^4({\bf x})$ is an order 4 vector. If $\D$ is cyclic, then ${\bf z}=(z\mid z')\in \D$. Note that ${\bf z}$ has zeros in all the binary positions, i.e. $z=(0,\ldots,0)$. Thus, $z'$ has 4 odd coordinates and two coordinates, say $z'_i$ and $z'_j$ equal to `2'. But note that $z'_i$ or $z'_j$ (or both) is obtained as the addition of two odd coordinates. Therefore, ${\bf x} - \sigma^4({\bf x})$ has weight less than 8, getting a contradiction.
\end{proof}

\bigskip

Now, we establish the main result of this section.

\begin{theorem}
If $\C'=\C_{r,t}'$ is a $\Z_2\Z_4$-additive extended 1-perfect code with $t\geq 3$, then $\C'$ is not $\Z_2\Z_4$-cyclic.
\end{theorem}

\begin{proof}
Consider the subcode $\C_0'=\{(x\mid 0,\ldots,0)\}$. If $\C'$ is $\Z_2\Z_4$-cyclic, then clearly $(\C_0')_X$ is a binary linear cyclic code. For every vector ${\bf v}=(v\mid 0,\ldots,0)$ of odd weight, we have that ${\bf v}$ must be at distance 1 from one codeword in $\C'$. Since no codeword ${\bf z}$ can have only an odd coordinate in the $\Z_4$ part (otherwise $2{\bf z}$ would have weight 2), it follows that $v$ is at distance 1 from a codeword in $(\C_0')_X$. Therefore $\C_0'$ must be an extended Hamming code. But such code cannot be cyclic unless it has length 4 \cite{Just}. The result then follows by Lemma \ref{DosConcrets}.
\end{proof}

%








\end{document}